\documentclass[10pt,a4paper]{article}
\usepackage[utf8]{inputenc}
\usepackage[english]{babel}
\usepackage{amsmath}
\usepackage{amsfonts}
\usepackage{amssymb}
\usepackage{amsthm}
\usepackage{bm}
\usepackage{color}
\usepackage{float}
\usepackage{dsfont}
\usepackage{titlesec}

\usepackage{tikz}	%include TikZ-engine
\usetikzlibrary{calc}

\usepackage[left=2cm,right=2cm,top=2cm,bottom=2cm]{geometry}

\numberwithin{equation}{section}

\newtheorem{thm}{Theorem}
\numberwithin{lem}{section}
\theoremstyle{remark}

\numberwithin{rem}{section} 

\newcommand{\R}{\mathbb{R}}
\newcommand{\Sp}{\mathbb{S}}
\newcommand{\eqdef}{\overset{def}{=}}	%eqdef sign

\newcommand{\supp}{\mathrm{supp}}

\newcommand{\footremember}[2]{%
    \footnote{#2} 
    \newcounter{#1}
    \setcounter{#1}{\value{footnote}}%
}

\title{An example of non-uniqueness for  \\  
the weighted Radon transforms\\ along hyperplanes in multidimensions}

\author{F.O. Goncharov\footremember{cmap}{CMAP, Ecole Polytechnique, CNRS, Universit\'{e} Paris-Saclay, 91128  Palaiseau, France; 
\newline \indent email: fedor.goncharov.ol@gmail.com} 
	\and	 
		R. G. Novikov$^*$\footnote{IEPT RAS, 117997  
			Moscow, Russia;
			\newline\indent email: roman.novikov@polytechnique.edu}}
		
\begin{document}
\maketitle
\abstract{We consider the weighted Radon transforms $R_W$ along hyperplanes in $\R^d, \, d\geq 3$, with strictly positive weights $W= W(x,\theta), \, x\in \R^d, \, \theta\in \Sp^{d-1}$.  We construct an example of such a transform with non-trivial kernel in the space of infinitely smooth compactly supported functions. In addition, the related weight $W$ is infinitely smooth almost everywhere and is bounded. Our construction is based on the famous example of non-uniqueness of J. Boman (1993) for the weighted Radon transforms in $\R^2$ and on a recent result of F. Goncharov and R. Novikov (2016).\\

\noindent \textbf{Keywords:} weighted Radon transforms, injectivity, non-injectivity\\

\noindent \textbf{AMS Mathematics Subject Classification:} 44A12, 65R32
}

\section{Introduction}\label{sect.introd}
We consider the weighted Radon transforms $R_W$, defined by the formulas:
\begin{align} \label{RW_def}
	&R_Wf(s,\theta) = \int\limits_{x\theta=s} W(x,\theta)
	f(x)\,  dx, \, (s,\theta)\in \R\times \Sp^{d-1}, \, x\in \R^d, \, d\geq 2,
\end{align}
where $W = W(x,\theta)$ is the weight, $f = f(x)$ is a test function on $\R^d$. 
\par We assume that $W$ is real valued, bounded and strictly positive, i.e.: 
\begin{align}\label{weight.cond.pos}
	W = \overline{W} \geq c > 0, \, W\in L^{\infty}(\R^d\times \Sp^{d-1}),
\end{align}
where $\overline{W}$ denotes the complex conjugate of $W$, $c$ is a constant.
\par If $W\equiv 1$, then $R_W$ is reduced to the classical Radon transform $R$ along hyperplanes in $\R^d$. This transform is invertible by the classical Radon inversion formulas; see \cite{radon1917}. 
\par If $W$ is strictly positive, $W\in C^{\infty}(\R^d\times \Sp^{d-1})$ and $f\in C_0^{\infty}(\R^d)$, then in \cite{beylkin1984inversion} the inversion of $R_W$ is reduced to solving a Fredholm type linear integral equation. Besides, in \cite{boman1987support} it was proved that $R_W$ is injective (for example, in $L_0^2(\R^d)$) if $W$ is real-analytic and strictly positive. In addition, an example of $R_W$ in $\R^2$ with infinitely smooth strictly positive $W$ and with non-trivial kernel $\mathrm{Ker}R_W$ in $C_0^{\infty}(\R^2)$ was constructed in \cite{boman1993example}. Here $C_0^{\infty}, \, L_0^2$ denote the spaces of functions from $C^{\infty}, \, L^2$ with compact support, respectively.
\par In connection with the most recent progress in inversion methods for weighted Radon transforms $R_W$, see \cite{goncharov2016iterative}.
\par We recall also that inversion methods for $R_W$ in $\R^3$ admit applications in the framework of emission tomographies (see \cite{goncharov2016analog}).

\par In the present work we construct an example of $R_{W}$ in $\R^d,\, d\geq 3$, with 
non-trivial kernel $\mathrm{Ker}R_W$ in $C_0^{\infty}(\R^d)$. The related $W$ satisfies \eqref{weight.cond.pos}. In addition, our weight $W$ is infinitely smooth almost everywhere on $\R^d\times \Sp^{d-1}$. 
%in $x$ on $\R^d$ and in $\theta$ on $\Sp^{d-1}\backslash M$, where 
%$M = \Sp^{d-3}$ (i.e., $M$ is a subsphere of $\Sp^{d-1}$). In particular, for $d=3$, the weight $W$ is smooth everywhere on $\R^3\times \Sp^2$ except the two directions $\theta= (0,0,\pm 1)$.
%\par In dimension $d=2$ the theory of non-injective Radon transforms is well developed; see, for example, \cite{boman1993example} and references therein. In particular, this theory \cite{boman1993example} includes the famous counterexample:
%\begin{equation}
%	P_{w_b}f_b \equiv 0,
%\end{equation} 
%where $P_{w_b}$ is the weighted ray (Radon) transform on the plane with weight $w_b$.
%In particular, in this example $w_b$ is real-valued, strictly positive, bounded and is %infinitely smooth on $\R^2\times \Sp^1$; the function $f_b$ belongs to $C_0^{\infty}(\R^2)$ (see Section~\ref{sect.weight.transf} for details).
%\par On the other hand, in dimension $d\geq 3$, to our knowledge, the counterexamples to injectivity of $R_W$, under assumptions \eqref{weight.cond.pos} (or similar assumptions), have not been given yet.  
\par In our construction we proceed from results of \cite{boman1993example} and  \cite{goncharov2016analog}.
%on the reduction of the weighted ray transforms $P_w$ to the weighted Radon transforms $R_W$ for $d=3$ (see Section~\ref{sect.weight.transf}). Combining these two works, we obtain the results presented below (see Section~\ref{sect.main.thm}).

\par In Section~\ref{sect.weight.transf}, in particular, we recall the result of \cite{goncharov2016analog}.
\par In Section~\ref{sect.bom.exmp} we recall the result of \cite{boman1993example}.
\par In Section~\ref{sect.main.thm} we obtain the main result of the present work.
%\par In Section~\ref{sect.main.ext} we extend the constructions of Section~\ref{sect.main.thm} for $d > 3$.

%\par In the case of $W\equiv 1, \, w\equiv 1$ the transforms $R_W, \, P_w$ are known as the classical Radon transform $R$ and the ray transform $P$, respectively.
%\par The question about the injectivity of the Radon transforms (in all dimensions $n\geq 2$) had been solved a long time ago. In particular, there exist exact and explicit inversion formulas for $R_W,\, (W\equiv 1)$; see, for example, the book \cite{natterer2001mathematics}.
%\par On the other hand, for $R_W$ with non-constant weights $W$ the situation is very different. For example, if $W$ is real-analytic and strictly positive, then the related transform $R_W$ is injective; see \cite{boman1987support}. 
%At the same time \cite{boman1993example}, it was shown that ...

\section{Relations between the Radon and the ray transforms}\label{sect.weight.transf} 
We consider also the weighted ray transforms $P_w$ in $\R^d$, defined by the formulas:
\begin{align}\label{weighted.ray.def}
	&P_wf(x,\theta) = \int\limits_{\R}w(x + t\theta, \theta) f(x + t\theta)\, dt,
	\, (x,\theta)\in T\Sp^{d-1},\\
	&T\Sp^{d-1} = \{(x,\theta) \in \R^d \times \Sp^{d-1} : x\theta = 0\}, \, d \geq 2,
\end{align}
where $w=w(x,\theta)$ is the weight, $f=f(x)$ is a test-function on $\R^d$.
\par We assume that $w$ is real valued, bounded and strictly positive, i.e.:
\begin{equation}
	w = \overline{w} \geq c > 0, \, w\in 
	 L^{\infty}(\R^d\times \Sp^{d-1}).
\end{equation}
We recall that $T\Sp^{d-1}$ can be interpreted as the set of all oriented rays in $\R^d$. In particular, if $\gamma = (x,\theta)\in T\Sp^{d-1}$, then
\begin{align}
	&\gamma = \{y\in \R^d : y = x + t\theta, \, t\in \R\},
\end{align}
where $\theta$ gives the orientation of $\gamma$.

\par We  recall that for $d=2$, transforms $P_w$ and $R_W$ are equivalent up to the following change of variables:
\begin{align}\label{R_WtoP_W}
	&R_Wf(s,\theta) = P_wf(s\theta, \theta^\perp), s\in \R, \, \theta\in \Sp^1,\\
	\begin{split} \label{Wtow}
	&W(x,\theta) = w(x,\theta^\perp), \, x\in \R^2, \, 
	\theta \in \Sp^1, 	\\ 
	&\theta^\perp = (-\sin\phi, \cos\phi)
	\text{ for } \, \theta = (\cos\phi, \sin\phi), \,  
	\phi\in [0,2\pi), 	
	\end{split}
\end{align}
where $f$ is a test-function on $\R^2$.
%\par  We recall that for $d=2$, in \cite{boman1993example} there were constructed the weight $w_b$ and the test-function $f_b$, such that:
%\begin{align}
%	&P_{w_b}f_b \equiv 0 \text{ on } T\Sp^{1}, \\
%	&w_b\in C^{\infty}(\R^2\times \Sp^1), \, 1\geq w_b \geq 1/2,  \\
%	&f_b\in C_0^{\infty}(\R^2), \, f\neq 0, \, \supp \, f_b\subset\overline{B^2}
%	=\{x\in \R^2 : |x|\leq 1\}.
%\end{align}
\par For $d=3$, the transforms $R_W$ and $P_w$ are 
related by the following formulas (see \cite{goncharov2016analog}):  
\begin{align}\label{reduction.formula1}
	&R_Wf(s,\theta) = \int\limits_{\R}
	P_wf(s\theta + \tau[\theta,\alpha(\theta)], \, \alpha(\theta))\, d\tau, \, (s,\theta)\in \R\times\Sp^2,\\ \label{reduction.formula2}
	&W(x,\theta) = w(x,\alpha(\theta)), \, x\in \R^3,\, \theta \in \Sp^2,\\
	\label{alpha_theta_3d}
	&\alpha(\theta) = \begin{cases}
		\dfrac{[\eta,\theta]}{|[\eta,\theta]|}, 
		\text{ if } \theta\neq \pm\eta,\\
		\text{ any vector }e\in \Sp^2, \text{ such that } e \perp \theta,	
		\text{ if } \theta = \pm \eta,
	\end{cases}
\end{align}
where $\eta$ is some fixed vector from $\Sp^2$, $[\cdot, \cdot]$ denotes the standard vector product in $\R^3$, 
$\perp$ denotes the orthogonality of vectors. 
Actually, formula \eqref{reduction.formula1} gives an expression for $R_Wf$ on $\R\times \Sp^2$ in terms 
of $P_wf$ restricted to the rays $\gamma = \gamma(x,\theta)$,  such that $\theta\perp \eta$, where $W$ and $w$ are related by \eqref{reduction.formula2}.

%Note that $W$ of \eqref{reduction.formula2} is defined for almost all $\theta\in \Sp^2$ using $w(\cdot,\alpha)$, where 
%$\alpha$ belongs only to the two-dimensional subspace $\Sigma_\eta = \{x\in \R^3 : \alpha\eta = 0\}$ of $\R^3$.
\par Below we present analogs of \eqref{reduction.formula1}-\eqref{reduction.formula2} for $d > 3$.
\par Let
\begin{align}\label{plane_d}
	&\Sigma(s,\theta) = \{x\in \R^d : x\theta = s\}, \, s\in \R, \, \theta\in \Sp^{d-1},\\
	\label{plane_2D}
	&\Xi(v_1,\dots, v_k) = \mathrm{Span}\{v_1,\dots, v_k\}, v_i\in \R^d, \, i=\overline{1,k}, 1\leq k \leq d,\\ \label{Msphere}
	&\Theta(v_1,v_2) = \{\theta\in \Sp^{d-1} : \theta \perp v_1, \theta\perp v_2\}
	\simeq \Sp^{d-3}, \, v_1, v_2 \in \R^d, \, v_1\perp v_2, \\
	\label{constant_basis_choice}
	&(e_1,e_2, e_3, \dots , e_d) 
	\text{ - be some fixed orthonormal, positively oriented basis in } \R^d.
\end{align} 
If $(e_1,\dots, e_d)$ is not specified otherwise, it is assumed that $(e_1,\dots, e_d)$ is the standard basis 
in $\R^d$.

%\par Without loss of generality we can assume that $P_wf$ is restricted to $\gamma = \gamma(x,\theta)$, where $\theta$ belongs to the two-dimensional subspace given by $L(e_1,e_2)$, i.e., $\theta \in \Sp^{d-1}\cap L(e_1,e_2)$.
%, where $w(\cdot, \alpha)$ is also given only for $\alpha$ from some fixed two-dimensional subspace $\Pi$ of $\R^d$.
\par For $d\geq 3$, the transforms $R_W$ and $P_w$ are related by the following formulas:
\begin{align}\label{RW_PW_d}
	&R_Wf(s,\theta) = \int\limits_{\R^{d-2}} P_wf(s\theta + 
	\sum\limits_{i=1}^{d-2}\tau_i \beta_i(\theta), \alpha(\theta))\, 
	d\tau_1\dots d\tau_{d-2}, \, (s,\theta)\in\R\times\Sp^{d-1},\\ \label{W_w_d}
	&W(x,\theta) = w(x,\alpha(\theta)), \, x\in \R^d, \, 
	\theta\in \Sp^{d-1},
\end{align}
where $\alpha(\theta), \, \beta_i(\theta), \, i = \overline{1,d-2}$, are defined as follows:
\begin{align}
	\label{direction_alpha}
	&\alpha(\theta) = \begin{cases} 
		\text{direction of one-dimensional intersection } 
		\Sigma(s,\theta)\cap \Xi(e_1,e_2), \text{ where}\\ 
		\text{the orientation of } \alpha(\theta) \text{ is chosen such that }
		\det(\alpha(\theta), \theta, e_3, \dots, e_d) > 0,
		\text{ if }\theta\not\in \Theta(e_1,e_2),
		\vspace{0.1cm}\\ 
		\text{any vector }e\in \Sp^{d-1}\cap \Xi(e_1,e_2), \text{ if } \theta\in \Theta(e_1,e_2),
	\end{cases}\\ 
	\label{basis_on_hypp}
	&(\alpha(\theta), \beta_1(\theta), \dots, \beta_{d-2}(\theta))
	\text{ is an orthonormal basis on }\Sigma(s,\theta),
\end{align}
and $\Sigma(s,\theta), \,\Theta(e_1,e_2)$ are given by \eqref{plane_d}, \eqref{Msphere}, 
respectively. Here, due to the condition $\theta\not\in \Theta(e_1,e_2)$:
\begin{equation}\label{intersect_dim}
\dim(\Sigma(s,\theta)\cap \Xi(e_1,e_2)) = 1.
\end{equation}
\par Formula \eqref{intersect_dim} is proved in Section~\ref{sect_proof_formula}.
\par Note that formulas \eqref{RW_PW_d}-\eqref{intersect_dim} are also valid for $d=3$. In this case these formulas are reduced to \eqref{reduction.formula1}-\eqref{alpha_theta_3d}, where $e_3 = -\eta$.
\par Note that, formula \eqref{RW_PW_d} gives an expression for $R_Wf$ on $\R\times \Sp^{d-1}$ in terms of $P_{w}f$ restricted to the rays $\gamma = (x,\alpha)$, such that $\alpha \in \Sp^{d-1}\cap \Xi(e_1,e_2)$.\\
\textbf{Remark 1.} In \eqref{direction_alpha} one can also write:
\begin{equation}\label{hodge_star}
	\alpha(\theta) = (-1)^{d-1} \star (\theta \wedge e_3\wedge\dots \wedge e_d),
	\text{ if } \theta\not\in \Theta(e_1,e_2),
\end{equation}
where $\star$-denotes the Hodge star, $\wedge$ - is the exterior product in $\Lambda^*\R^d$ (exterior algebra on $\R^d$); see, for example, Chapters 2.1.c, 4.1.c of  \cite{shigeyuki2001diff}.

\par Note that the value of the integral in the right hand-side of \eqref{RW_PW_d} does not 
depend on the particular choice of $(\beta_1(\theta), \dots, \beta_{d-2}(\theta))$ of \eqref{basis_on_hypp}. 
\par Note also that, due  to \eqref{reduction.formula2}, \eqref{alpha_theta_3d}, \eqref{W_w_d}, \eqref{direction_alpha}, the weight $W$ is defined everywhere on $\R^d\times \Sp^{d-1}, \, d\geq 3$. In addition,  
this $W$ has the same smoothness as $w$ in $x$ on $\R^d$ and in $\theta$ on $\Sp^{d-1}\backslash \Theta(e_1,e_2)$,
where $\Theta(e_1,e_2)$ is defined in \eqref{Msphere} and has zero Lebesgue measure on $\Sp^{d-1}$. 
%\begin{align}
%	&P_{w}f(x,\theta) = R_{W}f(s,\theta^\perp), 
%	W(x,\theta) = w(x,\theta^\perp), \, s = x\theta^\perp, \, 
%	x\in \R^2, \, \theta\in \Sp^1,\\
%	&\theta = (\cos\phi, \sin\phi), \, \theta^\perp = (-\sin\phi, \cos\phi), \, 
%	\phi \in [0,2\pi),
%\end{align}
%where $R_W, \, P_w, $ are defined by \eqref{RW_def}, \eqref{weighted.ray.def}, respectively. 

\section{Boman's example}\label{sect.bom.exmp}
\par  For $d=2$, in \cite{boman1993example} there were constructed a weight $W$ and a function $f$, such that:
\begin{align}\label{ray_b_props}
	&R_{W}f \equiv 0 \text{ on } \R\times \Sp^1, \\ \label{w_b_props}
	&1/2 \leq W \leq 1, W\in C^{\infty}(\R^2\times \Sp^1),  \\ \label{f_b_properties}
	&f\in C_0^{\infty}(\R^2), \, f\not\equiv 0, \, \supp \, f\subset\overline{B^2}
	=\{x\in \R^2 : |x|\leq 1\}.
\end{align}
In addition, as a corollary of \eqref{R_WtoP_W}, \eqref{Wtow}, \eqref{ray_b_props}-\eqref{f_b_properties}, we have that 
\begin{align}\label{ray_boman}
	&P_{w_0} f_0 \equiv 0 \text{ on } T\Sp^1,\hspace{5.2cm}\\
	&1/2 \leq w_0\leq 1, \, w_0 \in C^{\infty}(\R^2\times \Sp^1),\\
	&f_0\in C_0^{\infty}(\R^2), \, f_0\not\equiv 0, \, \supp \, f\subset\overline{B^2}
	=\{x\in \R^2 : |x|\leq 1\},
\end{align}
where
\begin{align}\label{ray_boman_weight}
	&w_0(x, \theta) = W(x, -\theta^\perp), \, x\in \R^2,\, \theta\in \Sp^1,\\
	\label{ray_boman_func}
	&f_0 \equiv f.
\end{align}

\section{Main results}\label{sect.main.thm}
Let
\begin{align}
	&B^{d} = \{x\in \R^d : |x| < 1\}, \\
	&\overline{B^d} = \{x\in \R^d : |x| \leq 1\},\\ \label{choice_eta}
	&(e_1,\dots, e_d) \text{ - be the canonical basis in } \R^d.
\end{align}
\begin{thm}\label{main.thm}
	There are $W$ and $f$, such that
	\begin{align}\label{RW_zero}
		&R_Wf \equiv 0 \text{ on } \R\times \Sp^{d-1},\\ \label{Wf_props}
		&W \text{ satisfies \eqref{weight.cond.pos}}, \, f\in C_0^{\infty}(\R^d), \, 
		f\not \equiv 0,
	\end{align}
	where $R_W$ is defined by \eqref{RW_def}, $d\geq 3$. 
	In addition, 
	\begin{equation}\label{almost_sur_smooth_prop}
		1/2 \leq W\leq 1, \, W \text{ is } C^{\infty}\text{-smooth on } \R^d\times (\Sp^{d-1}\backslash 
		\Theta(e_1,e_2)),
	\end{equation}
	where $\Theta(e_1,e_2)$ is defined by \eqref{Msphere}.	Moreover, weight $W$ and function $f$ are given by formulas \eqref{W_w_d}, \eqref{rayweight.def}-\eqref{psi.def} in terms of the J. Boman's weight $w_0$ and function $f_0$ of \eqref{ray_boman_weight}, \eqref{ray_boman_func}.
\end{thm}
\noindent \textbf{Remark 2.} According to \eqref{W_w_d}, \eqref{direction_alpha},  $W(x,\theta)$ for $\theta\in \Theta(e_1,e_2)$ can be specified as follows:
	\begin{equation}\label{W_overdefined}
		W(x,\theta) = W(x_1,\dots, x_d, \theta) \eqdef w_0(x_1,x_2, e_1), \, \theta\in \Theta(e_1,e_2), \, x\in \R^d.	
	\end{equation}

\begin{proof}[Proof of Theorem~\ref{main.thm}]
We define
\begin{align}\label{rayweight.def}
	&w(x,\alpha) = w(x_1,\dots,x_d, \alpha) \eqdef w_0(x_1,x_2,\alpha_1,\alpha_2),\\ \label{test_func_def}
	&f(x) = f(x_1,\dots,x_d) \eqdef \psi(x_3,\dots, x_d) f_0(x_1,x_2),\\ \nonumber
	&\text{for } 
	x = (x_1,\dots,x_d)\in \R^d,
	\, \alpha = (\alpha_1, \alpha_2, 0, \dots ,0)\in \Sp^{d-1}\cap \Xi(e_1,e_2)\simeq \Sp^1,
\end{align}
where 
\begin{align} 
	\label{psi.def}
	&\psi\in C_0^{\infty}(\R^{d-2}), \, \supp\,  \psi = 
	\overline{B^{d-2}} 	\text{ and } \psi(x) > 0 \text{ for } x\in B^{d-2}.
\end{align}
From \eqref{weighted.ray.def}, \eqref{ray_boman},  \eqref{rayweight.def}-\eqref{psi.def} it follows that:
	\begin{align}\label{PW_alpha_zero}
		\begin{split}
		P_wf(x,\alpha) &= \int\limits_{\R}
		w(x_1 +t\alpha_1 ,x_2 + t\alpha_2, x_3,\dots, x_d, \alpha)
		f(x_1 +t\alpha_1 ,x_2 + t\alpha_2, x_3, \dots, x_d)\,dt\\
		&= \psi(x_3, \dots, x_d)\int\limits_{\R}w_0(x_1 +t\alpha_1 ,x_2 + t\alpha_2, \alpha_1, \alpha_2)
		f_0(x_1 +t\alpha_1, x_2 + t\alpha_2)\, dt\\
		&= \psi(x_3,\dots, x_d) P_{w_0}f_0(x_1,x_2, \alpha_1,\alpha_2) = 0
		\text{ for any }\alpha = (\alpha_1,\alpha_2, 0,\dots , 0)\in \Xi(e_1,e_2)\cap \Sp^{d-1}\simeq \Sp^1. 
		\end{split}
	\end{align}
		Properties \eqref{RW_zero}-\eqref{almost_sur_smooth_prop} follow from \eqref{W_w_d}-\eqref{basis_on_hypp}, \eqref{hodge_star}, \eqref{w_b_props}, \eqref{f_b_properties}, \eqref{W_overdefined}, \eqref{rayweight.def}.
	\par Theorem~\ref{main.thm} is proved.
\end{proof}

\section{Proof of formula \eqref{intersect_dim}}\label{sect_proof_formula}
\par Note that 
\begin{equation}
	\dim(\Xi(e_1,e_2)) + \dim(\Sigma(s,\theta)) = d+1 > d,
\end{equation}
which implies that the intersection $\Sigma(s,\theta)\cap \Xi(e_1,e_2)$ is one of the following:
\begin{enumerate}
 \item The intersection is the one dimensional line $l = l(s,\theta)$:
\begin{equation}\label{line_intersection}
	l(s,\theta) = \{x\in \R^d : x = x_0(s,\theta) + \alpha(\theta) t ,\, t\in \R\}, \, 
	\alpha(\theta)\in \Sp^2,
\end{equation}
	where $x_0(s,\theta)$ is an arbitrary point of $\Sigma(s,\theta)\cap \Xi(e_1,e_2)$,  the orientation of $\alpha(\theta)$ is chosen such that:
	\begin{equation}\label{direct.alpha}
		\det(\alpha(\theta), \, \theta, \, e_3, \dots, e_d) > 0.
	\end{equation}
	Condition \eqref{direct.alpha} fixes uniquely the direction of
	 $\alpha(\theta)$ of \eqref{line_intersection}.
	\par Formulas \eqref{plane_d}, \eqref{plane_2D}, \eqref{Msphere} imply that \eqref{direct.alpha} can hold
	if and only if $\theta\not\in \Theta(e_1,e_2)$.
	%\par In particular, for $d=3$ the aforementioned choice of $\alpha(\theta)$ is equivalent to the following formula:
	%\begin{equation}
	%	\alpha(\theta) = \dfrac{[e_3,\theta]}{|[e_3,\theta]|},
	%	\, \theta\neq \pm e_3,
	%\end{equation}
	%where $[\cdot, \cdot ]$ denotes the standard vector product in $\R^3$ and 
	%$\eta = e_3$ is also a normal to $\Pi$.
	\item The intersection is the two-dimensional plane $\Xi(e_1,e_2)$. Formulas \eqref{plane_d}, \eqref{plane_2D} imply that it is the case if and only if 
	\begin{equation}
		s = 0,\, \theta\perp e_1, \, \theta\perp e_2.
	\end{equation}
	\item The intersection is an empty set. Formulas \eqref{plane_d},
	 \eqref{plane_2D} imply that it is the case if and only if
	\begin{equation}
		s \neq 0,\, \theta\perp e_1, \, \theta\perp e_2.
	\end{equation}
\end{enumerate}
\par Note that 
\begin{align}\label{cases_occurence}
 \begin{split} 
	&\text{cases 2 and 3 occur if and only if } \theta \perp e_1, \, 
	\theta \perp e_2, \text{ i.e., }\theta\in \Theta(e_1,e_2).
 \end{split}
\end{align}
\par This completes the proof of formula \eqref{intersect_dim}.

\section{Acknowledgments}
This work is partially supported by the PRC $n^{\circ}$ 1545 CNRS/RFBR: \'{E}quations quasi-lin\'{e}aires, probl\`{e}mes inverses et leurs applications.

\bibliographystyle{alpha}

\end{document}